\documentclass[11pt]{amsart}

\usepackage{epigamath-voisin-runin}

\usepackage[english]{babel}

\numberwithin{equation}{section}

\usepackage[shortlabels]{enumitem}

\usepackage[T1]{fontenc}
\usepackage{mathtools}
\usepackage{graphicx}
\usepackage{tikz}
\usetikzlibrary{chains}
\usepackage{tcolorbox}
\usepackage{mathabx}
\usepackage{scrextend}
\PassOptionsToPackage{pdfusetitle,pagebackref,colorlinks}{hyperref}
\usepackage{bookmark}

\usepackage{amsmath}
\usepackage{amsthm}
\usepackage{amssymb}

\usepackage[all]{xy}

\usepackage[OT2,T1]{fontenc}

\newtheorem{thm}{Theorem}[section]

\newtheorem{lem}[thm]{Lemma}
\newtheorem{cor}[thm]{Corollary}

\newtheorem*{claim}{Claim}

\theoremstyle{definition}

\theoremstyle{remark}
\newtheorem{remark}[thm]{Remark}
\newtheorem{ex}[thm]{Example}

\def\isolow{\vbox to 0pt{\vss\hbox{$\scriptstyle\sim$}\vskip-2pt}}
\newcommand{\congpf}{\xrightarrow{\isolow}}
\newcommand{\longcongpf}{\xrightarrow{\;\isolow\;}}
\newcommand{\lra}{\longrightarrow}

\DeclareRobustCommand\longtwoheadrightarrow
    {\relbar\joinrel\twoheadrightarrow}
    
\DeclareRobustCommand\longtwoheadleftarrow
    {\joinrel\twoheadleftarrow\joinrel\relbar}   

\DeclareRobustCommand\longhookrightarrow
    {\lhook\joinrel\relbar\joinrel\rightarrow}
    \renewcommand{\cong}{\simeq}

\newcommand{\kt}{\mathcal T}
\newcommand{\kn}{\mathcal N}
\newcommand{\ke}{\mathcal E}
\newcommand{\kf}{\mathcal F}
\newcommand{\ks}{\mathcal S}
\newcommand{\ko}{\mathcal O}
\newcommand{\QQ}{\mathbb Q}

\newcommand{\ZZ}{\mathbb Z}

\newcommand{\PP}{\mathbb P}
\newcommand{\GG}{\mathbb G}

\DeclareMathOperator{\CH}{CH}

\DeclareMathOperator{\pr}{pr}
\DeclareMathOperator{\Gr}{Gr}
\DeclareMathOperator{\Alb}{Alb}

\DeclareMathOperator{\id}{id}
\DeclareMathOperator{\Bl}{Bl}
\DeclareMathOperator{\res}{res}

\DeclareSymbolFont{cyrletters}{OT2}{wncyr}{m}{n}
\DeclareMathSymbol{\Sha}{\mathalpha}{cyrletters}{"58}

\newcommand{\TBC}[1]{}

\setlist[enumerate,1]{label={\rm(\roman*)}, ref=\roman*} 

\YearArticle{2023} \EpigaArticleNr{5} \ReceivedOn{December 5, 2022}
\InFinalFormOn{April 28, 2023}
\AcceptedOn{May 21, 2023}

\title{Chow groups of surfaces of lines in cubic fourfolds}
\titlemark{Chow groups of surfaces of lines in cubic fourfolds}

\author{D.~Huybrechts}
\address{
Mathematisches Institut, Universit\"at Bonn, Endenicher Allee 60, 53115 Bonn, Germany}
\email{huybrech@math.uni-bonn.de}

\authormark{D.~Huybrechts}

\AbstractInEnglish{
The surface of lines in a cubic fourfold intersecting a fixed line splits motivically into two parts, one of which resembles a K3 surface. We define the analogue of the Beauville--Voisin class and study the push-forward map to the Fano variety of all lines with respect to the natural splitting of the Bloch--Beilinson filtration introduced by Mingmin Shen and Charles Vial.}



\acknowledgement{This research was supported by the ERC Synergy Grant HyperK, Grant agreement ID 854361, and the Hausdorff Center for Mathematics, Bonn (Germany's Excellence Strategy--EXC-2047/1--390685813).}

\dedication{To Claire Voisin, with admiration}

\begin{document}

%%%%%%%%%%%%%%%%%%%%%%%%%%%%%%%
% Title page
%%%%%%%%%%%%%%%%%%%%%%%%%%%%%%%

\maketitle

\begin{prelims}

\DisplayAbstractInEnglish

\bigskip

\DisplayKeyWords

\medskip

\DisplayMSCclass

\end{prelims}

%%%%%%%%%%%%%%%%%%%%%
% Table of Contents
%%%%%%%%%%%%%%%%%%%%%

\newpage

\setcounter{tocdepth}{1}

\tableofcontents

%%%%%%%%%%%%%%%%%%%%%
% Content begins here
%%%%%%%%%%%%%%%%%%%%%

\section{Introduction}

Fano varieties $F=F(X)$ of lines $L\subset X$ in smooth complex cubic fourfolds $X\subset\PP^5$ are hyperk\"ahler manifolds of dimension four and as such have many properties and features in common with K3 surfaces. In particular, the Chow ring 
$\CH^\ast(F)$ of the Fano variety is of great interest and has been investigated from various angles by Beauville \cite{B1}, Voisin \cite{V2,V1}, M.\ Shen and Vial \cite{SV}, J.\ Shen, Yin, and Zhao \cite{SYZ,SY}, and many others. 

\subsection{} The first goal of this note is to study zero-dimensional cycles on $F$ %,so the group $\CH^4(F(X))=\CH_0(F(X))$,
 from the perspective of distinguished surfaces contained in $F$. More specifically, we continue our investigation \cite{H1} and consider the surface $F_{L_0}\subset F$ of all lines $L\subset X$ intersecting a given fixed line $L_0\subset X$. The surfaces $F_{L_0}$ and their Chow groups are of great interest in their own right, but here, in order to gain a better understanding of the Chow group of zero-dimensional cycles $\CH^4(F)=\CH_0(F)$ on $F$, 
we  first have a closer look at the push-forward map $\CH_0(F_{L_0})\to\CH_0(F)$. 

The target and source of this map both come with a natural
decomposition, and our first result explains how they are related.

\begin{itemize}
\item  There exists a natural involution $\iota\colon F_{L_0}\congpf F_{L_0}$, and its quotient $ F_{L_0}\to D_{L_0}\coloneqq F_{L_0}/_\pm$
 describes a quintic surface $D_{L_0}\subset\PP^3$  with $16$ nodes. Homologically trivial zero-cycles on $D_{L_0}$ can be identified with  $\iota$-invariant homologically trivial zero-cycles on $F_{L_0}$; \textit{i.e.}\ $\CH_0(D_{L_0})_{\hom}$ is
the first summand of the $\iota$-eigenspace decomposition 
\begin{equation}\label{eqn:1stdecomp}
\CH_0\left(F_{L_0}\right)_{\hom}=\CH_0\left(F_{L_0}\right)_{\hom}^+\oplus\CH_0\left(F_{L_0}\right)^-.
\end{equation}
The second summand gives back the interesting part of the Chow group of the cubic itself, namely $\CH_0(F_{L_0})^-\cong \CH_1(X)_{\hom}$ via the Fano correspondence.

\item  On the Fano side, the group $A\coloneqq \CH_0(F(X))$ splits
as 
\begin{equation}\label{eqn:2nddecomp}
A=A_4\oplus A_2\oplus A_0.
\end{equation}
 Here, $A_0$ is generated by a distinguished class $c_F$, the analogue of the Beauville--Voisin class on a
K3 surface, and $A_4$ is the deepest part of
the Bloch--Beilinson filtration. This decomposition was studied
in depth by Shen and Vial \cite{SV}; see Section~\ref{sec:BBFano} for more details and references.
\end{itemize}

Our first result clarifies the link between the two decompositions 
(\ref{eqn:1stdecomp}) and (\ref{eqn:2nddecomp}). The result confirms
 certain aspects of the Bloch--Beilinson philosophy in the present situation.

\pagebreak
\begin{thm}\label{thm:main}
Assume $L_0\subset X$ is a general line in a smooth cubic fourfold $X\subset\PP^5$.
\begin{enumerate}
\item\label{thm:main-1} The push-forward map of invariant homologically trivial cycles
\begin{equation}\label{eqn:thmmain1}
\CH_0\left(D_{L_0}\right)_{\hom}\cong \CH_0\left(F_{L_0}\right)^+_{\hom}\to A
\end{equation}
takes image in $A_4\subset A$. For  very general $X$  and  very general
$L_0\subset X$, the image is non-trivial.
Furthermore, the image of the composition 
$$
\CH_0\left(D_{L_0}\right)\lra A\longtwoheadrightarrow A/A_4=A_2\oplus A_0
$$
is spanned by the class  $c_{F}(L_0)\coloneqq 2c_F-[L_0]_2$,
where $[L_0]_2$ denotes the $A_2$-component of\, $[L_0]$.

\item \label{thm:main-2} The composition of the push-forward map and the
projection $A\twoheadrightarrow A_2\oplus A_0$
describes an isomorphism
$$
\CH_0\left(F_{L_0}\right)^-\longcongpf A_2$$
with its
inverse induced by the map $[L]\mapsto (-1/2)[F_L]|_{F_{L_0}}$. Furthermore, if $L_0$ is very general, the image of\, $\CH_0(F_{L_0})^-\to A$ is not contained in $A_2$.
\end{enumerate}
\end{thm}

\subsection{} The second goal of the paper is to investigate the K3 nature of
the  `negative half\,' $F_{L_0}^-$ of the surface $F_{L_0}$. The surface $F_{L_0}$ `decomposes' into two parts, and one component behaves very much like a K3 surface. On the level of Chow groups, the two parts correspond to the
$\iota$-invariant and the $\iota$-anti-invariant cycles, and
for rational Chow motives we have the decomposition ${\mathfrak h}(F_{L_0})\cong{\mathfrak h}(F_{L_0})^+\oplus {\mathfrak h}(F_{L_0})^-$; see the discussion in \cite[Section~4.6]{H1}. The K3  nature of the anti-invariant part is best seen by the observation
that $H^2(F_{L_0},\ZZ)^-_{\pr}$ is a Hodge structure of K3 type of rank $22$, which,
up to a factor and a Tate twist, is in fact Hodge isometric to  the primitive middle cohomology
$H^4(X,\ZZ)_{\pr}$ of the cubic fourfold;
\textit{cf.}  \cite[Theorem~0.2]{H1}.

One distinguished feature, \textit{cf.}   \cite{BV}, of the Chow ring of a complex K3 surface $S$ is that
despite $\CH^2(S)$ being big, all intersection products
$\alpha_1\cdot\alpha_2$ of classes $\alpha_1,\alpha_2\in\CH^1(S)$
 are proportional. They are multiples of the Beauville--Voisin class $c_S\in\CH^2(S)$, which
is realised by any point contained in a rational curve. 
Transplanted in our context, this triggers the question whether 
products $\alpha_1\cdot\alpha_2$ of two anti-invariant 
classes $\alpha_1,\alpha_2\in \CH^1(F_{L_0})^-$ are all multiples of a distinguished
class $c_{L_0}\in \CH^2(F_{L_0})$, which would necessarily be invariant. Our second result
shows that this is the case for primitive classes and after push-forward to $F$.

\begin{thm}\label{thm:main2}
For a general line $L_0\subset X$ in a smooth cubic fourfold, 
the composition 
$$
\CH^1\left(F_{L_0}\right)_{\pr}^-\times\CH^1\left(F_{L_0}\right)_{\pr}^-\lra \CH^2\left(F_{L_0}\right)^+\lra A
$$
takes image in $\ZZ\cdot c_{F}(L_0)$ with $c_F(L_0)=2c_F-[L_0]_2\in A_2\oplus A_0\subset A$ as above.
\end{thm}

This suggests that we should view
$$\CH^\ast\left(F_{L_0}^-\right)\coloneqq\ZZ\oplus \CH^1\left(F_{L_0}\right)^-_{\pr}
\oplus\left(\CH^2\left(F_{L_0}\right)^-\oplus \ZZ\cdot c_{F}(L_0)\right)$$
as the Chow ring of $F_{L_0}^-$, \textit{i.e.}\ of the `K3 half\,' of the surface $F_{L_0}$.
Coming back to Theorem~\ref{thm:main}, the natural isomorphism
$\CH_0(F_{L_0})^-\cong A_2$ highlights the K3 character of the middle
part $A_2$ of the decomposition (\ref{eqn:2nddecomp}).

\begin{remark}\label{rem:Intro} One could ask whether 
 $\CH^2(F_{L_0}^-)=\CH^2(F_{L_0})^-\oplus \ZZ\cdot c_{F}(L_0)$ can be realised simply as the image of $\CH_0(F_{L_0})\to\CH_0(F)$, so that the hyperk\"ahler fourfold
$F$ singles out the `K3 half\,' of $F_{L_0}$. However, this fails due to the 
slightly counter-intuitive non-triviality of \eqref{eqn:thmmain1} 
in Theorem~\ref{thm:main}\eqref{thm:main-1}. 

In the opposite direction, one could wonder whether \eqref{eqn:thmmain1} is maybe injective. This would have the effect
that $c_F(L_0)$ lifts uniquely to a class $c_{L_0}\in\CH_0(F_{L_0})^+$
such that
$$
\ZZ\oplus \CH^1\left(F_{L_0}\right)^-_{\pr}
\oplus\left(\CH^2\left(F_{L_0}\right)^-\oplus \ZZ\cdot  c_{L_0}\right)\subset\CH^*\left(F_{L_0}\right)
$$
is a subring, \textit{i.e.}\ such that $\alpha_1\cdot\alpha_2\in \ZZ \cdot c_{L_0}$ for
all $\alpha_1,\alpha_2\in \CH^1(F_{L_0})_{\pr}^-$. The class $c_{L_0}$ on $F_{L_0}$ would then be the geometric realisation of the Beauville--Voisin class
for the K3 half $F_{L_0}^-$; see Remark~\ref{rem:final} for further comments.
\end{remark}

The main tool to prove these results is the map $\psi\colon \CH_0(F)\to \CH_2(F)$, $L\mapsto [F_L]$, which has been studied and used before; see \cite{SV,SYZ,SY,V1}. Further restricting to surfaces contained in $F$ allows one to transform zero-cycles on $F$ into zero-cycles on surfaces. However, since $\psi$ annihilates the deepest part of the Bloch--Beilinson filtration
$A_4\subset\CH_0(F)$, this technique does not provide
 control over all zero-cycles, on the Fano variety
or on surfaces contained in it.

\vskip\baselineskip

\noindent {\large\bf Acknowledgement.}
\medskip

I wish to thank Claire Voisin for helpful comments and inspiration in general. I am particularly indebted to Charles Vial for many helpful comments. In particular, Remark~\ref{rem:ref}\eqref{rem:ref-2} follows his suggestions.

%%%%%%%%%%%%%%%
\section{Zero-cycles on hyperk\"ahler fourfolds}
 It is expected that the Chow groups of an arbitrary smooth projective
variety $Z$ are endowed with natural (Bloch--Beilinson) filtrations
$$\cdots \subset F^2\subset F^1\subset \CH^k(Z)$$
that enjoy certain functoriality properties. Also, the filtrations
should be finite, more precisely $F^i=0$ for $i>k$, and the behaviour of algebraic correspondences on the graded parts $F^i/F^{i+1}$ should be determined by their action on the corresponding pieces
of the Hodge structure of $Z$.

For example, any correspondence
$\Gamma_\ast\colon \CH_0(Z_1)\to \CH_0(Z_2)$ should respect the filtration, \textit{i.e.}\   $\Gamma_*(F^i(Z_1))\subset F^i(Z_2)$, and the induced
map ${\Gr}^i(\Gamma_\ast)\colon (F^i/F^{i+1})(Z_1)\to  (F^i/F^{i+1})(Z_2)$
should be zero if and only if the cohomological action $\Gamma^\ast\colon H^0(Z_2,\Omega_{Z_2}^i)\to H^0(Z_1,\Omega_{Z_1}^i)$ is trivial. This is often seen as a generalisation of Bloch's conjecture; see \cite[Conjecture~23.22]{VoisinHodge}.

There are various suggestions for such filtrations proposed
by H.~Saito \cite{HSaito}, S.~Saito \cite{SSaito}, Nori \cite{Nori}, and Voisin \cite{V3},
but proving the desired properties in general seems out of reach for now.

\begin{ex}\label{ex:surface}
For zero-cycles on a smooth projective surface $S$, the various requirements   determine the Bloch--Beilinson filtration completely: concretely, $F^1\CH_0(S)$ is the subgroup of all cycles
of degree zero or, equivalently, of those that are algebraically or, still equivalently, homologically trivial,
while $F^2\subset F^1$ is the kernel of the Albanese morphism
$F^1\to{\Alb}(S)$.
\end{ex}

\subsection{}
For a hyperk\"ahler manifold $Z$  of dimension four, 
 the conjectural filtration on zero-cycles  would be of the form $$0\subset F^4\subset F^3\subset F^2\subset F^1\subset F^0=A\coloneqq\CH_0(Z).$$
Recall that by Roitman's result \cite{Roit} the group $A$ is torsion-free, and,
by using the divisibility of Jacobians of curves contained in $Z$, the kernel of the degree
map $\deg\colon A\to\ZZ$ is known to be divisible.

Bloch's conjecture applied to the identity correspondence $[\Delta]_\ast={\id}$ leads to the following picture. The vanishings $H^0(Z,\Omega_Z^3)=0=H^0(Z,\Omega_Z^1)$ should imply $F^4=F^3$ and $F^2=F^1$, and the non-vanishings $H^0(Z,\Omega_Z^4)\ne0 \ne H^0(Z,\Omega_Z^2)$ should give $F^4\ne0$ and $F^3\ne F^2$. So the conjectural Bloch--Beilinson filtration in this case would actually look like
\begin{equation}\label{eqn:conjfiltr}
0\subsetneqq F^4=F^3\subsetneqq F^2=F^1\subsetneqq F^0=A=\CH_0(Z).
\end{equation} In full generality, only $F^2=F^1$ is
rigorously defined. Also note that a generalisation of an argument of Mumford for surfaces shows that $F^1$ is not trivial and in fact not representable by any finite-type scheme; see \cite[Theorem~22.15]{VoisinHodge}.

Taking inspiration from the case of K3 surfaces \cite{BV}, Beauville \cite{B1} asked whether the hyperk\"ahler structure of $Z$ gives rise
to a certain splitting of the conjectural filtration (\ref{eqn:conjfiltr}). This would  lead to a decomposition
$$A=A_4\oplus A_2\oplus A_0,$$
with $F^4=F^3=A_4$, $F^2=F^1=A_4\oplus A_2$, and $\deg\colon A_0\congpf \ZZ$. The generator of degree one of $A_4$ is called the Beauville--Voisin class $c_Z\in A$.

Shen and Vial \cite{SV} suggested to use a certain lift $\tilde q\in \CH^2(Z\times Z)$ of the class $q_Z\in S^2H^2(Z,\QQ)\subset H^4(Z\times Z,\QQ)$ induced
by the Beauville--Bogomolov--Fujiki pairing as the analogue of the Poincar\'e bundle for abelian varieties to define the decomposition as
$$A_k\coloneqq \left\{\alpha\in A\mid \exp(\tilde q)(\alpha)\in \CH^k(Z)\right\}.$$

An alternative and more geometric definition was proposed by Voisin \cite{V1}:
One first introduces the orbit ${\rm O}_x\subset Z$ of any point $x\in Z$
as the set of all points  that are rationally equivalent to $x$. The orbit turns out to be a countable union of closed subsets of dimensions at most two. Then, 
$A_0\subset A$ and $A_2\oplus A_0$ should be generated by those points with $\dim {\rm O}_x\geq 2$ and $\dim {\rm O}_x\geq 1$, respectively.

\subsection{}\label{sec:BBFano} Typically, conjectures concerning hyperk\"ahler manifolds (of dimension four) are first checked for the Hilbert scheme $S^{[2]}$ of a K3 surface $S$ and for the Fano variety $F=F(X)$ of lines contained in a smooth cubic fourfold $X\subset \PP^5$. This has been undertaken for Beauville's question by
Beauville \cite{B1} himself,  by Voisin \cite{V2,V1}, by M.\ Shen and Vial \cite{SV}, and more recently by J.\ Shen, Yin, and Zhao \cite{SY,SYZ}.  We briefly review some of the results for the Fano variety
$F=F(X)$ and its group of zero-cycles $A=\CH_0(F)$.

Firstly, there is a natural candidate for the Beauville--Voisin class, \textit{i.e.}\
the generator of $A_0$, namely the unique class  $c_F\in \CH_0(F)$ satisfying
$$6\cdot c_F= g\cdot[C_x].$$
Here, $g\in \CH^1(F)$ is the Pl\"ucker polarisation and 
$C_x\coloneqq\{L\mid x\in L\}\subset F$, which is a curve for the general point $x\in X$.
Observe that, since the cubic $X$ is unirational, the class of the curve
$[C_x]\in \CH^3(F)$ does not depend on the choice of the general point $x$.
The class $c_F$ indeed has the desired multiplicative property: according to a result of Voisin \cite[Theorem~0.4]{V2}, the degree four component of any polynomial expression involving only 
the Chern classes ${\rm c}_2(F)$ and ${\rm c}_4(F)$ and classes $\alpha\in \CH^1(F)$  is a multiple of $c_F$.

Secondly, one exploits the Voisin map $f\colon F\dashrightarrow F$ and its induced action on the Chow group $f_\ast\colon A\to A$.
Recall that the Voisin map is the endomorphism of degree 16
that maps a line  of the first type $L$ to the residual line $L'$ of the intersection
$L\subset P_L\cap X$ with the unique plane $P_L\cong\PP^2$ tangent to $L\subset X$; \textit{i.e.}\ $P_L\cap X=2L\cup L'$.  Shen and Vial 
\cite[Theorem~21.9]{SV} proved that the induced action $f_\ast$ on $\CH_0$
has eigenvalues $4$, $-2$, and $1$ and declare the corresponding eigenspaces 
 to be $A_4$, $A_2$, and $A_0=\ZZ\cdot c_F$.

Thirdly, mapping a line $L$ to the surface $F_L$ induces a map
\begin{equation}\label{eqn:psi}
\psi\colon A\lra \CH_2(F),\quad [L]\longmapsto [F_L], 
\end{equation}
and one defines $A_4$ as the kernel of $\psi$. Alternatively, $A_4$ can also be described as the homologically trivial part of the subgroup
generated by all triangles, \textit{i.e.}\ by sums $[L_0]+[L_1]+[L_2]$ of triples of lines  $L_0,L_1,L_2\subset X$
spanning a plane. Yet another possibility is to describe $A_4$  as the kernel
of the Fano correspondence $\varphi\colon A=\CH_0(F)\to \CH_1(X)$
mapping the class of a point in $F$ corresponding to a line in $X$ to the class of that line. In fact, $\varphi$ factors through an isomorphism: 
\begin{equation}\label{eqn:CH1X}
\varphi\colon A\longtwoheadrightarrow A/A_4\cong A_2\oplus A_0\longcongpf \CH_1(X).\end{equation}

It is \textit{a priori}
not clear that the various definitions of the deepest
part of the Bloch--Beilinson filtration
\begin{equation}\label{eqn:Defi1}
A_4=\left\{\alpha\mid f_\ast\alpha=4\cdot\alpha\right\}=\ker(\psi)=\ker(\varphi)=\left\langle[L_0]+[L_1]+[L_2]\right\rangle_{\hom}
\end{equation}
or of its complement $$A_2\oplus A_0=\left\{\alpha\mid f_*\alpha=-2\cdot\alpha\right\}\oplus
\left\{\alpha\mid f_*\alpha=\alpha\right\}
 =\left\langle [L]\mid \dim {\rm O}_{L}\geq 1\right\rangle$$
describe the same subgroups. The often geometrically quite intricate arguments have been
provided by Shen and Vial \cite{SV} and Voisin \cite{V1}. Note that
$A_4$ and $A_2$ are both non-trivial, but this is not automatic. In fact,
both parts should be thought of as `big', \textit{i.e.}\ non-representable,
although $A_2$ is surface-like while $A_4$ `lives' in dimension four;  
see below for more on this.

Of the many results in the comprehensive monograph of Shen and Vial \cite{SV}, the following will be important for our discussion:

\begin{enumerate}
\item The first concerns the multiplication $\CH^2(F)\times\CH^2(F)\to \CH^4(F)=A$. For a triangle $L_0\cup L_1\cup L_2\subset X$, the classes $[F_{L_i}]\in \CH^2(F)$ satisfy, \textit{cf.}  \cite[Proposition~20.7]{SV}, 
\begin{equation}\label{eqn:F1stSV}
\left[F_{L_1}\right]\cdot \left[F_{L_2}\right]=6\cdot c_F+[L_0]-[L_1]-[L_2].
\end{equation}
\item The second is about the image of the distinguished class $c_F$ under
(\ref{eqn:psi}). 
According to \cite[Lemma~A.5]{SV}, one has
\begin{equation}\label{eqn:F2ndSV}
3\cdot \psi(c_F)=g^2-{\rm c}_2(\ks_F)=(1/8)(5g^2+{\rm c}_2(F)),
\end{equation}
where $\ks_F$ is the universal subbundle on $F\subset \GG(1,\PP^5)$.

\item For every homologically trivial class $\gamma\in{\Im}(\psi)$, which according to \cite[Section~21]{SV} is equivalent to being an algebraically trivial class in $\CH^2(F)$, and 
every class $\alpha\in\CH^1(F)$ we have
\begin{equation}\label{eqn:gamma}
\alpha^2\cdot\gamma\in A_2.
\end{equation}
This follows from combining $A_2=\CH^1(F)_0^{\cdot 2}\cdot\CH^2(F)_2$,
see \cite[Proposition~22.2]{SV},  $\CH^1(F)_0=\CH^1(F)$, see \cite[Theorem~2]{SV},
$\CH^2(F)_2=V^{2}_{-2}$, see \cite[Theorem~21.9(iii)]{SV}, and
$V^2_{-2}={\mathcal A}_{\hom}={\Im}(\psi)_{\hom}$, see \cite[Definition~20.1 and Proposition~21.10]{SV}.
\end{enumerate}

\section{Special surfaces}
Apart from the surfaces $F_L\subset F=F(X)$, there are other interesting types
of surfaces in $F$. Most notably, there are the surface $F'\subset F$ of 
lines of the second type and the surfaces $F(Y)\subset F$ of lines contained
in general hyperplane sections $Y=X\cap \PP^4$.
Before turning to cycles on these surfaces, let us recall a few geometric facts. 

 All three surfaces $F_L$, $F'$, and $F(Y)$ are smooth surfaces of general type.\footnote{For $F'$ to be smooth, one needs to assume $X$ general.}
Their basic numerical invariants are as follows:
\begin{enumerate}[label=(\arabic*)]
\item $p_g(F(Y))=10$ and $q(F(Y))=5$;  see \cite[Chapter~5]{HuyCubic} for references. 
\item $p_g(F')=449$ and $q(F')=0$; see \cite{GK} and \cite[Section~6.4]{HuyCubic}.
\item $p_g(F_{L})=5$ and $q(F_L)=0$. In fact, $\pi_1(F_L)=\{1\}$, 
see \cite[Lemma~1.2 and Appendix]{H1}.
\end{enumerate}

\subsection{} Voisin proved \cite[Example~3.7]{VoisinLagr} that the surfaces $F(Y)\subset F$ are Lagrangian; \textit{cf.}  \cite[Lemma~6.4.5]{HuyCubic}. In other words, the pull-back map $H^0(F,\Omega^2_F)\to H^0(F(Y),\Omega_{F(Y)}^2)$  is zero.

The Bloch--Beilinson filtration for zero-cycles on the surface $F(Y)$ is of the form
$$
0\subset F^2\subset F^1\subset \CH_0(F(Y)),
$$
where the Albanese map
$F^1/F^2\congpf {\Alb}(F(Y))$ is an isomorphism and $F^2$ is big, \textit{i.e.}\ not repre\-sentable; see Example~\ref{ex:surface}.

As always for the push-forward map from a surface, $\CH_0(F(Y))\to A$ respects the filtration. However, since $F(Y)\subset F$ is Lagrangian, the Bloch--Beilinson conjecture predicts that the image of $F^2\CH_0(F(Y))\to A_4\oplus A_2=F^2\subset A$ is actually contained in $F^4=F^3=A_4$. This prediction can be confirmed as follows:\footnote{Thanks to C.\ Voisin for the argument.\label{note1}} As~the Fano correspondences for $X$ and for the hyperplane section $Y\subset X$ are compatible, the natural
diagram 
$$\xymatrix{\CH_0(F(Y))\ar[d]_-{\varphi_Y}\ar[r]&A\ar[d]^-{\varphi_X}\\
\CH_1(Y)\ar[r]&\CH_1(X)}$$
commutes. By the definition of the Bloch--Beilinson filtration for the
surface $F(Y)$, we know that $F^2\CH_0(F(Y))$ is the kernel
of the Albanese map. Since ${\Alb}(F(Y))\cong\CH_1(Y)_{\hom}$, 
by a result of Beauville and Murre, \textit{cf.}  \cite[Corollary~5.3.16]{HuyCubic},  this shows that $F^2\CH_0(F(Y))$ is contained
in the kernel of $\varphi_Y$. Therefore, its image in $A$ is contained in the kernel of $\varphi_X$, which is $F^4=A_4\subset A$ by (\ref{eqn:Defi1}).

Note that for the general hyperplane section $Y\subset X$, the surface $F(Y)\subset F$ is not a constant cycle surface; \textit{i.e.}\ the image of $\CH_0(F(Y))\to A$ is not of rank one. Indeed, otherwise the generator $c_Y$ would stay constant,
since $Y$ varies in the projective space $|\ko(1)|$, but as every line in $X$ is contained in some hyperplane section, this would imply the triviality of $A_4\oplus A_2$.

A stronger statement is in fact true. Namely, for the very general hyperplane section $Y$, 
the map $F^2\CH_0(F(Y))\to A_4$ is not trivial. 
Here is a sketch of the argument:\footref{note1}
The surfaces $F(Y)$ cover $F$ since
every line in $X$ is contained in some hyperplane section. In fact, the family of all $F(Y)$ can be seen as a $\PP^3$-bundle $\kf\to F$ with a projection
$\kf\to|\ko(1)|$. 
 Restricting to a general two-dimensional family $\pi\colon\kf_{\PP^2}\to\PP^2\subset|\ko(1)|$ gives
a fourfold with a dominant morphism  $\kf_{\PP^2}\to F$ inducing an injection $H^{4,0}(F)\,\hookrightarrow H^{4,0}(\kf_{\PP^2})$. However, over a dense open subset $U\subset \PP^2$, 
we have $\Omega^4_{\kf_{U}}\cong\pi^\ast\Omega_U^2\otimes\Omega_\pi^2$,
but if $F^2\CH_0(F(Y))\to A$ were zero, then the component in $\Omega_\pi^2$ would be trivial.

\begin{remark} Although  for the very general hyperplane section,  the surface $F(Y)\subset F$ is not a constant cycle surface,  special ones might be.
Imitating the study of constant cycle curves in ample linear systems on K3 surfaces, see \cite{Huyccc}, it should be interesting to study those in more detail.\footnote{Thanks to E.\ Sert\"oz for a related question.} In fact, Voisin \cite[Lemma~2.2]{V2} shows that the surface of lines contained in a hyperplane section
with five nodes is rational (and singular) and hence a constant cycle surface.
\end{remark}

\subsection{} Now let $X$ be general such that the surface $F'\subset F$ of lines of the second type is smooth. Then $F^2=F^1\subset \CH_0(F')$,
which the push-forward map sends to $F^2=A_4\oplus A_2\subset A$. 
Since $F'\subset F$ is not Lagrangian, see \cite[Section~6.4.4]{HuyCubic},
the projection to the graded part gives a non-trivial
map $F^2\to A_4\oplus A_2\twoheadrightarrow A_2$. In fact,
according to a result of Shen and Vial \cite[Proposition~19.5]{SV}, the surface $F'$ avoids 
$A_4$ and covers $A_2$. More precisely,
\begin{equation}\label{eqn:F2F}
{\Im}\left( \CH_0(F')\lra A\right)=A_2\oplus A_0\subset A.
\end{equation}
Again by Mumford's argument, using $p_g(F')>1$ and that $F'\subset F$ is not Lagrangian,  the kernel and image of the push-forward map are not representable. 
However, a geometric understanding of those cycles on $F'$ that
become rationally trivial on $F$ is not available. Also, unlike the case $\CH_0(F_{L_0})^-\subset\CH_0(F_{L_0})$ studied in this paper, there does
not seem to be any distinguished
subgroup of $\CH_0(F')$ that maps isomorphically onto~$A_2$.

As was observed by Shen and Vial
\cite[Theorem~3]{SV}, (\ref{eqn:F2F}) also implies $A_4\ne0$. Indeed, otherwise
$\CH_0(F)$ would be concentrated on the surface $F'$, which by Bloch--Srinivas \cite{BS} would contradict $H^{4,0}(F)\ne0$.

\begin{remark}
Observe that (\ref{eqn:F2F}) fits Voisin's description of $A_2\oplus A_0$ (and was in fact used for its proof). Indeed, the rational endomorphism 
$f\colon F\dashrightarrow F$ is resolved by a blow-up of $F$
in $F'\subset F$. Thus, for a line of the second type $L\in F'$, the class
$f_\ast[L]$ is realised by all points $L'$ in the image of the exceptional line 
of the blow-up $\tau\colon {\Bl}_{F'}(F)\to F$ over
$L$, which clearly satisfy $\dim {\rm O}_{L'}\geq 1$. Hence, by Voisin's description, we have $f_*[L]\in A_2\oplus A_0$ and, therefore, $[L]\in A_2\oplus A_0$. For the other inclusion observe that the image
 of the exceptional divisor of $\tau$ is an ample divisor in $F$ which therefore intersects any curve. Hence, every line $L'\in F$ with $\dim {\rm O}_{L'}\geq 1$, and
 so in particular $[L']\in A_2\oplus A_0$ according to Voisin, 
 gives rise to a class of the form $f_\ast[L]$ for some $L\in F'$.
 \end{remark}

\section{Lines intersecting a fixed line}
Let us now fix a general line $L_0\subset X$ and consider the surface
 $F_{L_0}\subset F=F(X)$ of lines intersecting $L_0$. 
 Then the quotient by the standard involution defines a morphism
 $\pi\colon F_{L_0}\twoheadrightarrow D_{L_0}$ onto a quintic surface $D_{L_0}\subset\PP^3$.
 The 16 fixed points, \textit{i.e.}\ the 16 lines $L\subset X$ with $\overline{LL_0}\cap X=2L\cup L_0$, give rise to 16 ordinary double points of $D_{L_0}$; see 
 the original \cite{VoisinGT} or the discussion in \cite[Section~6.4.5]{HuyCubic} or in
 \cite{H1}.

The involution $\iota$ acts on $\CH_0(F_{L_0})$, and we define
$$
\CH_0\left(F_{L_0}\right)^\pm\coloneqq\{\alpha\mid\iota^\ast\alpha=\pm\alpha\}\subset\CH_0\left(F_{L_0}\right).
$$
Then the pull-back map induces isomorphisms
$$
\CH_0\left(D_L\right)_{\hom}\longcongpf\CH_0\left(F_L\right)_{\hom}^+\quad\text{ and }\quad\CH_0\left(F_{L_0}\right)_{\hom}^-\cong \CH_0\left(F_{L_0}\right)_{\hom}/\CH_0(D_L)_{\hom}.
$$
Moreover, 
$\CH_0(F_{L_0})^-=\CH_0(F_{L_0})_{\hom}^-=\ker(\pi_\ast)={\Im}(1-\iota^\ast)$; see \cite[Section~4.1]{H1}.

Of course, $\CH_0(D_{L_0})/\CH_0(D_{L_0})_{\hom}\cong \ZZ$, but this isomorphism comes without a canonical split; \textit{i.e.}\ 
no analogue of the Beauville--Voisin class exists for the surface $D_{L_0}$. 
Theorem~\ref{thm:main}\eqref{thm:main-1} can be seen as a replacement: the image of $\CH_0(D_{L_0})$ in $A/A_4\cong A_2\oplus A_0$ is generated by a distinguished class depending on $L_0$.

\subsection{} Before entering the  proof of Theorem~\ref{thm:main}\eqref{thm:main-1},
we observe that the correspondence
$$D_{L_0}\longtwoheadleftarrow F_{L_0}\longhookrightarrow  F$$
induces the trivial map
$H^0(F,\Omega_F^2)\to H^0(D_{L_0},\Omega_{D_{L_0}}^2)$ since
the restriction of the holomorphic two-form on $F$ to $F_{L_0}$ is anti-invariant; see \cite[Section~2.3]{H1}.
In particular, according to the Bloch--Beilinson philosophy, one expects the map
$$
\CH_0\left(D_{L_0}\right)_{\hom}=\left(F^2/F^3\right)\left(D_{L_0}\right)\lra \left(F^2/F^3\right)(F)=A_2
$$
to be trivial. Equivalently, the composition of pull-back and push-forward should map the group of zero-cycles $\CH_0(D_{L_0})_{\hom}=\CH_0(F_{L_0})_{\hom}^+$  on $D_{L_0}$ to $F^3=A_4\subset A$. In this sense, the first
assertion Theorem~\ref{thm:main}\eqref{thm:main-1} can be seen as a confirmation
of the Bloch--Beilinson philosophy.

\begin{proof}[Proof of Theorem~\ref{thm:main}(\ref{thm:main-1})] A point $t\in D_{L_0}$ corresponds
to a pair of lines $L_1,L_2\subset X$ such that $L_0\cup L_1\cup L_2$ forms a triangle, \textit{i.e.}\ all three lines are contained in a single plane. 

Under push-forward, the class $[t]\in \CH_0(D_{L_0})$ of the point $t\in D_{L_0}$ is then mapped to $[L_1]+[L_2]\in A$
and, after composing further with $\psi\colon A\to \CH_2(F)$, $[L]\mapsto[F_L]$,  to
\begin{equation}\label{eqn:FL1L2}
\left[F_{L_1}\right]+\left[F_{L_2}\right]=\varphi\left(h^3\right)-\left[F_{L_0}\right].
\end{equation}
Here, $h$ is the hyperplane class on $X$, $h^3$ can be represented
by the plane $\PP^2=\overline{L_0L_1L_2}$, and $\varphi\colon\CH_1(X)\to \CH_2(F)$ is the Fano correspondence. As $\varphi(h^3)-[F_{L_0}]$ is independent of the point $t\in D_{L_0}$
and $\CH_0(D_{L_0})_{\hom}$ is spanned by classes of the form $[t]-[t']$, $t,t'\in D_{L_0}$, we find that the image of $\CH_0(D_{L_0})_{\hom}\to A$ is contained in the
kernel of $\psi$, which is $A_4$.

Next let us show that the image of the composition $\CH_0(D_{L_0})\to
A\to A/A_4\cong A_2\oplus A_0$ is spanned by the class $c_F(L_0)=2c_F-[L_0]_2$, where $[L_0]_2$ is the degree two part of
$$
[L_0]=[L_0]_4+[L_0]_2+[L_0]_0\in A_4\oplus A_2\oplus A_0.
$$
In order to prove this, observe that 
the image of $3c_F-([L_0]_2+[L_0]_0)=c_F(L_0)$ under
the injection $\bar\psi\colon A/A_4\,\hookrightarrow \CH_2(F)$
 can be computed by means of (\ref{eqn:F2ndSV}) as
 $$
 \psi\colon A\lra \CH_2(F),\quad c_F(L_0)\longmapsto  g^2-{\rm c}_2(\ks_F)-\left[F_{L_0}\right].
 $$
The latter
equals $\varphi(h^3)-[F_{L_0}]$ by \cite[Proposition~6.4.1]{HuyCubic}, and to conclude
we apply (\ref{eqn:FL1L2}), which proves $\varphi(h^3)-[F_{L_0}]=[F_{L}]+[F_{\iota(L)}]$ for any $L\in F_{L_0}$. Hence, the classes
$c_F(L_0)$ and $[L]+[\iota(L)]$ modulo $A_4$ have the same image under the injection $\bar\psi$. Therefore, the image of
$$
\ZZ\cong\frac{\CH_0\left(D_{L_0}\right)}{\CH_0\left(D_{L_0}\right)_{\hom}}\lra A/A_4\cong A_2\oplus A_0
$$
is indeed $\ZZ\cdot c_F(L_0)$.\medskip

Before proving the rest of Theorem~\ref{thm:main}\eqref{thm:main-1}, we note the following.

\begin{claim} Assume $\rho(F)>1$. Then the class  $c_F(L_0)$ is contained in the image of\, $\CH_0(D_{L_0})\to A$.
  \end{claim}

\begin{proof}
It suffices to prove the existence of an invariant class $\eta\in \CH_0(D_{L_0})^+$ with a non-trivial push-forward contained in $A_2\oplus A_0$. By Lemma~\ref{lem:A2squares},
the proof of which is independent of the rest of the discussion here, this holds
for the restriction of the square $\eta=(\alpha|_{F_{L_0}})^2$ of any primitive class $0\ne\alpha\in \CH^1(F)_{\pr}$.
\end{proof}

It remains to prove that $\CH_0(D_{L_0})_{\hom}\to A_4$ is not zero
for the very general $X$ and the very general line $L_0\subset X$. Clearly,
it suffices to show this for one cubic, but using the above claim, we will in fact show it for all cubics with $\rho(F)>1$.

So we assume $\rho(F)>1$ and suppose that for the very general choice of $L_0$ and hence for all $L_0$, the map $\CH_0(D_{L_0})_{\hom}\to A_4$ is zero. 
Then, using that $\CH_0(D_{L_0})\to A/A_4$ has rank one and the
fact that by the claim $c_F(L_0)\in A_2\oplus A_0$ is contained in the image of the push-forward map,
 all of $\CH_0(D_{L_0})$ would map to $A_2\oplus A_0$ (with its image generated by $c_F(L_0)$). Now consider a fixed point
$L\in F_{L_0}$ of the covering involution $F_{L_0}\to D_{L_0}$, \textit{i.e.}\ a line $L\subset X$ such that
\begin{equation}\label{eqn:L0L1}
2L\cup L_0=\PP^2\cap X
\end{equation} for some plane $\PP^2\subset\PP^5$. Then $2[L]\in \CH_0(D_{L_0})\subset \CH_0(F_{L_0})^+$, and for $L$
as a point in $F$, one then has $[L]\in A_2\oplus A_0$. However, for
every line $L\in F$, there exists a line $L_0$ satisfying (\ref{eqn:L0L1}), and
for $L$ general $L_0$ also is general (and so the surface $F_{L_0}$ is smooth). Hence,  the triviality of $\CH_0(D_{L_0})_{\hom}\to A_4$ for the general $L_0$ would prove that all points in $F$ define classes in $A_2\oplus A_0$, \textit{i.e.}\ $A_4=0$, which is absurd.
\end{proof}

\begin{remark}
The last part of the argument also shows  that in general the class of a fixed point $L\in F_{L_0}$ of the covering involution $\iota$ is not mapped to the distinguished class $(1/2)c_F(L_0)\in A_2\oplus A_0\subset A_4\oplus A_2\oplus A_0$. Hence, at least for generic choices, none of the 16 fixed points of $\iota$ is a candidate for the Beauville--Voisin class $c_{L_0}\in \CH_0(F_{L_0})^+$ in Remark~\ref{rem:Intro}; see also Remark~\ref{rem:final}.
\end{remark}

\begin{remark}
The Beauville--Voisin class on a K3 surface is by definition of degree one, while
the analogue $c_F(L_0)\in A_2\oplus A_0$ is of degree two. There are two reasons for this. First, the projection $F_{L_0}\to D_{L_0}$ is of degree two, and thus
the push-forward of any class on $F_{L_0}$ that is pulled back from $D_{L_0}$
has even degree. Second, the intersection pairing $(\alpha_1|_{F_{L_0}}.\alpha_2|_{F_{L_0}})$ of any two primitive classes $\alpha_1,\alpha_2\in\CH^1(F)_{\pr}$ is even; see \cite[Corollary~1.7]{H1}.
\end{remark}

\subsection{}
The proof of Theorem~\ref{thm:main}\eqref{thm:main-2} starts with the
following key technical lemma, which is based on an excess
intersection computation.

\begin{lem}\label{lem:Key} We consider the composition 
  $$
  \beta\colon \CH_0\left(F_{L_0}\right)\lra \CH_0(F)\stackrel{\psi}{\lra} \CH_2(F)\stackrel{\res}{\lra} \CH_0\left(F_{L_0}\right)
  $$
  of the push-forward map $\CH_0(F_{L_0})^-\to\CH_0(F)$, the map
$\psi\colon\CH_0(F)\to\CH_2(F)$, $[L]\mapsto[F_L]$, and the restriction map
${\res}\colon \CH_2(F)\to\CH_0(F_{L_0})$. Then,  $\beta=-2\cdot{\id}$.
\end{lem}

\begin{proof} The group $\CH_0(F_{L_0})^-$ is generated by
classes of the form $[L_1]-[L_2]$, where  $L_1\in F_{L_0}$ is any point and $L_2=\iota(L_1)$ is its image
under the covering involution of $F_{L_0}\to D_{L_0}$.
The map $\beta$ sends such a class to the restriction $([F_{L_1}]-[F_{L_2}])|_{F_{L_0}}$. 

First observe that the push-forward to $A$
turns $([F_{L_1}]-[F_{L_2}])|_{F_{L_0}}$ into $([F_{L_1}]-[F_{L_2}])\cdot [F_{L_0}]$, which by (\ref{eqn:F1stSV})
is indeed $-2([L_1]-[L_2])\in A$. We need to show that this equality holds before pushing forward, \textit{i.e.}\ $([F_{L_1}]-[F_{L_2}])|_{F_{L_0}}=-2([L_1]-[L_2])\in \CH_0(F_{L_0})^-$, and for this we have to find a geometric
interpretation of $[F_{L_i}]|_{F_{L_0}}$. 

The intersection $F_{L_1}\cap F_{L_0}$ is the set of all lines simultaneously intersecting $L_1$ and $L_0$. Hence,
$$
F_{L_1}\cap F_{L_0}=\{L_2\}\sqcup C_{1}.
$$
Here, $C_{1}\coloneqq\{L\mid x_1\in L\}$, with $x_1$ the point of intersection of $L_1$ and $L_0$. A similar statement holds for $F_{L_2}\cap F_{{L_0}}$. Hence,
$$
\left(\left[F_{L_1}\right]-\left[F_{L_2}\right]\right)|_{F_{L_0}}=[L_2]-[L_1]+E_1-E_2,
  $$
where $E_1$ and $E_2$ are the contributions from the excess intersections $C_{1}$ and $C_{2}$. We will show that $E_1-E_2=[L_2]-[L_1]\in\CH_0(F_{L_0})$.

To compute $E_1,E_2$ we recall that the excess
intersection $[S]|_{T}\in \CH_0(S_2)$ of two smooth surfaces $S,T\subset F$ intersecting in a smooth curve $C=S\cap T$ is the push-forward of
${\rm c}_1(\ke)$, where $\ke$ is the line bundle $\kn_{S/F}|_C/\kn_{C/T}\cong
\kn_{T/F}|_C/\kn_{C/S}
$; \textit{cf.}  \cite[Theorem~9.2]{Fulton} or \cite[Section~13.3]{EH}.
In other words, $\ke$ is part of a commutative diagram of short exact sequences
$$\xymatrix@R=18pt@C=18pt@M=8pt{\kn_{C/T}\ar@{^(..>}[r]&\kn_{S/F}|_C\ar@{->>}[r]&\ke\\
\kt_T|_C\ar@{^(->}[r]\ar@{->>}[u]&\kt_F|_C\ar@{->>}[r]\ar@{->>}[u]&\kn_{T/F}|_C\ar@{->>}[u]\\
\kt_C\ar@{^(->}[r]\ar@{^(->}[u]&\kt_S|_C\ar@{->>}[r]\ar@{^(->}[u]&\kn_{C/S}\rlap{.}\ar@{^(..>}[u]
}$$
Hence, as a line bundle on $C$, one has 
\begin{eqnarray*}
\ke&\cong&\det(\kn_{S/F})|_C\otimes\kn_{C/T}^\ast\\
&\cong&\omega_F^\ast|_C\otimes\omega_S|_C\otimes\omega_T|_C\otimes\omega_C^\ast.
\end{eqnarray*}
In our situation, we let $S=F_{L_i}$ and $T=F_{L_0}$. Their canonical bundles 
are $\pi_i^\ast\ko(1)$ and $\pi_0^\ast\ko(1)$, where the $\pi_i\colon F_{L_i}\to D_{L_i}\subset\PP^3$, $i=0,1,2$, denote
the projections. Hence, the excess contribution for the intersection
$[F_{L_i}]|_{F_{L_0}}$ comes from the line bundle 
\begin{equation}\label{eqn:Ei}
\ke_i\cong\pi_i^\ast\ko(1)|_{C_{i}}\otimes\pi_0^\ast\ko(1)|_{C_{i}}\otimes\omega_{C_{i}}^\ast.
\end{equation}

\begin{claim}If\, $\ko_F(1)$ denotes the Pl\"ucker polarisation on $F$  and the two lines $L_0,L_i$ are considered as points in $C_i$, then $$\ke_i\cong \ko_F(2)|_{C_i}\otimes \ko_{C_i}(-L_0-L_i)\otimes\omega_C^\ast.$$
\end{claim}
  
\begin{proof}To prove the claim it suffices to show that $\pi_j^\ast\ko(1)|_{C_i}\cong\ko_F(1)|_{C_i}\otimes\ko_{C_i}(-L_j)$  for $j=0$ and $j=i$. 
It is known that $\ko_F(1)|_{F_{L_j}}\cong\pi_j^\ast\ko(1)\otimes\ko_{F_j}(C_i)$;  see
the original \cite[Section~3, Lemma~2]{VoisinGT} or 
\cite[Equation~(4.5) in Remark~6.4.13]{HuyCubic}.
Thus, we need to prove that $\ko_{F_{L_j}}(C_i)|_{C_i}\cong\ko_{C_i}(L_j)$. To see this, observe that the
blow-up of $F_{L_j}$ in the point $L_j$ comes with the projection $q\colon {\Bl}_L(F_{L_j})\to L_j$, so that $C_i$ is the fibre over $x_i\in L_j$. In particular, the linear equivalence of $\ko_{F_{L_j}}(C_i)$ is independent of the point $x_i\in L_j$. As for two distinct points $x\ne x'\in L_j$, the two curves $C_x$ and $C_{x'}$
only intersect in the point  $L_j\in F_{L_j}$ and do so transversally, this concludes the proof of the claim.\end{proof}

The claim immediately shows $E_1-E_2=[L_2]-[L_1] \in \CH_0(F_{L_0})$, 
which concludes the proof of the lemma.
\end{proof}

\begin{remark}\label{rem:Proofrefined}
With the notation in the above proof, we have actually shown that
$$
\psi([L_1])|_{F_{L_0}}=[L_2]+E_1=[L_2]+2g\cdot C_1-[L_0]-[L_1]-[\omega_{C_1}].
$$
Since $\omega_{C_1}\cong(\omega_{F_{L_0}}\otimes\ko(C_1))|_{C_1}\cong(\pi_0^\ast\ko(1)\otimes\ko(C_1))|_{C_1}\cong\ko_F(1)|_{C_1}$, this gives
$$
\psi([L_1]+[L_2])|_{F_{L_0}}=2(g\cdot C-[L_0])\in \CH_0\left(F_{L_0}\right),
$$
where $C\subset F_{L_0}$ is any curve linearly equivalent to $C_1\subset F_{L_0}$ or, equivalently, to $C_2\subset F_{L_0}$.
\end{remark}

\begin{cor} The push-forward map
  $$
  \CH_0\left(F_{L_0}\right)^-\,\longhookrightarrow A
  $$
and its composition with the projection $A\twoheadrightarrow A/A_4\cong A_2\oplus A_0$
\begin{equation}\label{eqn:FA2inj}
\CH_0\left(F_{L_0}\right)^-\,\longhookrightarrow A\longtwoheadrightarrow A/A_4
\end{equation} are both injective. \qed
\end{cor}

\begin{proof}[Proof of Theorem~\ref{thm:main}(\ref{thm:main-2})] 
It remains to prove that for a very general  line $L_0$, the image of the push-forward
$\CH_0(F_{L_0})^-\to A$ is not contained in $A_2$ and that the composition (\ref{eqn:FA2inj}) induces an isomorphism $\CH_0(F_{L_0})^-\congpf A_2\subset A/A_4$.\

For the first assertion, we choose a line $L_0$ such that for the decomposition 
$$
[L_0]=[L_0]_4+[L_0]_2+[L_0]_0\in A_4\oplus A_2\oplus A_0,
$$
we have $[L_0]_4\ne0$. Then
we view  $L_0$ as a point in $F_{L_0}$ and consider its image $\iota(L_0)\in F_{L_0}$ under the involution. Their difference defines a class
$[\iota(L_0)]-[L_0]\in \CH_0(F_{L_0})^-$. By the definition of the Voisin endomorphism  $f\colon F\dashrightarrow F$, the push-forward
of this class is $f_\ast[L_0]-[L_0]\in A$, which by the definitions of $A_4$, $A_2$, and $A_0$ equals
$3([L_0]_4-[L_0]_2)$. This proves that the push-forward of $[\iota(L_0)]-[L_0]\in \CH_0(F_{L_0})^-$ is not contained in $A_2\oplus A_0$.

The surjectivity of the composition (\ref{eqn:FA2inj}) follows from (\ref{eqn:CH1X}), showing
$A_2\congpf \CH_1(X)_{\hom}$ via the Fano correspondence, and the
isomorphism $\CH_0(F_{L_0})^-\congpf\CH_1(X)_{\hom}$, obtained
as the composition of the push-forward and the Fano correspondence; \textit{cf.}  \cite[Theorem~0.4]{H1}.
\end{proof}

\begin{remark}\label{rem:ref}\leavevmode
  \begin{enumerate}
    \item\label{rem:ref-1} Since the  image of the other projection 
\begin{equation}\label{eqn:projA4}
\CH_0\left(F_{L_0}\right)^-\,\longhookrightarrow A\longtwoheadrightarrow A_4
\end{equation}
contains the component $[L_0]_4\in A_4$ and since $A_4$ is not generated by points in a single surface, the image of (\ref{eqn:projA4}) does indeed depend on $L_0$.
It is natural to wonder whether the composition (\ref{eqn:projA4}), similarly to $\CH_0(D_{L_0})_{\hom}\to A_4$ in Remark~\ref{rem:Intro}, is actually injective for general $L_0$.

\item\label{rem:ref-2} As pointed out by Charles Vial, for $[L_0]=c_F$ the map (\ref{eqn:projA4}) is actually trivial and in particular not injective. To prove this one first notices that for $[L_0]=c_F$, the endomorphism $\delta\colon\CH_0(F)\to\CH_0(F)$,
$\alpha\mapsto[F_{L_0}]\cdot \psi(\alpha)=(1/3)(g^2-{\rm c}_2(\ks_F))\cdot \psi(\alpha)$ preserves $A_2$. Using the notation of \cite{SV}, this follows from $g,{\rm c}_2(\ks_F)\in\CH^2(F)_0$, see \cite[Theorem~21.9]{SV},
$\CH^2(F)_2=\psi(A_4\oplus A_2)=\psi(A_2)$, see \cite[Proposition~21.10]{SV},
and $\CH^2(F)_0\cdot\CH^2(F)_2=A_2$, see \cite[Proposition~22.3]{SV}.
Then by Lemma~\ref{lem:Key},  $\delta$ maps the image
$\alpha_4+\alpha_2\in A_4\oplus A_2$ of a class in $\CH_0(F_{L_0})^-$
 to $-2\cdot(\alpha_4+\alpha_2)$. Since $\delta(\alpha_4+\alpha_2)=\delta(\alpha_2)\in A_2$, this implies $\alpha_4=0$.

 In fact, by Lemma~\ref{lem:Key} and the above proof of Theorem~\ref{thm:main}\eqref{thm:main-2}, the push-forward then defines an isomorphism
$\CH_0(F_{L_0})^-\congpf A_2$, and $\delta|_{A_2}=-2\cdot{\id}$. Note that this shows in particular that
$A_2$ is spanned by classes $[L_1]-[L_2]$ forming a triangle with a fixed line $L_0$ with $[L_0]=c_F$.

\item\label{rem:ref-3} Following a suggestion of Charles Vial, see also \cite[Theorem~20.2(ii)]{SV}, we observe that $A_4$ is spanned by
the images of all the maps (\ref{eqn:projA4}) for varying $L_0$. Indeed, since $A_4\oplus A_2$ is spanned by classes of the form $[L_1]-[L_2]$ and since
for any two lines $L_1,L_2$, there exists a line $L$ intersecting both, by writing
$[L_1]-[L_2]=([L_1]-[L])-([L_2]-[L])$ one reduce to the case that
$L_1,L_2$ are planar. In this case, if $L_0$ denotes the residual line of $L_1\cup L_2\subset \overline{L_1L_2}\cap X$, then $L_1=\iota(L_2)$ and,
therefore, $([L_1]-[L_2])_4$ is contained in the image of (\ref{eqn:projA4}).
\end{enumerate}
  \end{remark}

\subsection{} As the first step towards a proof of Theorem~\ref{thm:main2},
we recall from \cite[Theorem~0.2]{H1}, see also \cite[Theorem~3]{Izadi} and \cite[Theorem~4.7 and Corollary~4.8]{Shen}, that the restriction of line bundles induces
an isomorphism
\begin{equation}\label{eqn:oldPicresult}
\CH^1(F)_{\pr}\longcongpf\CH^1\left(F_{L_0}\right)^-_{\pr}.
\end{equation}
Here, the primitive parts on the two sides are defined with respect to the Pl\"ucker polarisation
$\ko_F(1)$  on $F$ and its restriction to $F_{L_0}$. Thus, to establish
Theorem~\ref{thm:main2}, it suffices to show that
for any $\alpha\in\CH^1(F)_{\pr}$, the class $\alpha^2\cdot[F_{L_0}]$
is a multiple of $c_F(L_0)=2c_F-[L_0]_2$.

The next step can be seen as a warm-up. The result was used already in the proof of Theorem~\ref{thm:main}\eqref{thm:main-1}.

\begin{lem}\label{lem:A2squares}
For any class $\alpha\in \CH^1(F)$, the push-forward
of\, $(\alpha|_{F_{L_0}})^2\in \CH_0(F_{L_0})$ to $F$ is a class in $A_2\oplus A_0$.
\end{lem}

\begin{proof}
The assertion is equivalent to $\alpha^2\cdot[F_{L_0}]\in A_2\oplus A_0$.
To prove this, recall that the two classes $\psi([L_0])=[F_{L_0}]$ and
$\psi(c_F)=(1/3)(g^2-{\rm c}_2(\ks_F))$ in $\CH^2(F)$ are homologically equivalent; \textit{cf.}  \cite[Proposition~6.4.1]{HuyCubic}.

Hence, $[F_{L_0}]-\psi(c_F)$ is contained in the homologically trivial
part of ${\Im}(\psi)$. Therefore, by (\ref{eqn:gamma}) we have $\alpha^2\cdot([F_{L_0}]-\psi(c_F))\in A_2$. Now, we apply 
a result of Voisin \cite[Theorem~0.4]{V2} and use
the well-known fact that ${\rm c}_2(\ks_F)$ is a linear combination of ${\rm c}_2(\kt_F)$ and $g^2$, \textit{cf.}  \cite[Proposition~6.4.1]{HuyCubic}, to conclude that 
 $3\alpha^2\cdot\psi(c_F)=\alpha^2\cdot(g^2-{\rm c}_2(\ks_F))\in A_0$.
Hence,  $\alpha^2\cdot [F_{L_0}]\in A_2\oplus A_0$.
\end{proof}

\begin{proof}[Proof of Theorem~\ref{thm:main2}] 
We apply Lemma~\ref{lem:A2squares} to an arbitrary anti-invariant primitive class $\delta\in \CH^1(F_{L_0})^-_{\pr}$, which by (\ref{eqn:oldPicresult}) can be written as $\alpha|_{F_{L_0}}$ for a unique primitive class $\alpha$ on $F$. Then the push-forward of $\delta^2\in \CH^2(F_{L_0})^+$ is contained in $A_2\oplus A_0$. 
As by Theorem~\ref{thm:main}\eqref{thm:main-1} the image of the composition $\CH^2(D_{L_0})\to A\twoheadrightarrow A/A_4=A_2\oplus A_0$ is spanned by $c_F(L_0)$, this
shows that the push-forwards of $\delta^2$ and $c_F(L_0)$ are proportional.

To conclude, note that all numerical intersection products $(\alpha_1|{F_{L_0}}.
\alpha_2|{F_{L_0}})$ for primitive classes $\alpha_1,\alpha_2\in \CH^1(F)_{\pr}$ are even; see \cite[Corollary~1.7]{H1}. This implies that $\alpha_1|_{F_{L_0}}\cdot \alpha_1|_{F_{L_0}}\in\CH^2(F_{L_0})^+$ is contained in the index two subgroup $\CH^2(D_{L_0})\subset\CH^2(F_{L_0})^+$.
\end{proof}

\begin{remark}
The result cannot be extended to the whole  
$\CH^1(F_{L_0})$. Indeed, the push-forward of $g^2|_{F_{L_0}}$, \textit{i.e.}\
$g^2\cdot[F_{L_0}]$, is not a multiple of $c_F(L_0)$, for
by \cite[Lemma~18.2]{SV} we have
\begin{eqnarray*}
  g^2\cdot\left[F_{L_0}\right]&=&[f(L_0)]-4[L_0]+24 c_F\\
&=&(4[L_0]_4-2[L_0]_2+[L_0]_0)-4[L_0]_4-4[L_0]_2-4[L_0]_0+24 c_F\\
&=&-6[L_0]_2-3[L_0]_0+24c_F,
\end{eqnarray*}
which can be a multiple of $c_F(L_0)$ only if $[L_0]_2=0$. By virtue of
\cite[Theorem~3.4]{SYZ}, the latter is in fact equivalent to $[L_0]=c_F$, so excluded for general $L_0$.

Also, extending to just the full anti-invariant part is not possible. Indeed,
as used before, the restriction of the Pl\"ucker polarisation satisfies
$\ko_F(1)|_{F_{L_0}}\cong\pi^\ast_0\ko(1)\otimes\ko(C_x)$ for any point $x\in L_0$. Hence, we have $(g_{F_{L_0}}-\iota^\ast(g_{F_{L_0}}))^2=C_x\cdot C_x+\iota(C_x\cdot C_x)-2 (C_x\cdot \iota(C_x))$ on $F_{L_0}$, which under
 push-forward to $F$ becomes
$[L_0]+f_\ast[L_0]-2(C_x\cdot \iota(C_x)=[L_0]+f_\ast[L_0]-2(6c_F-2[L_0])$,
which for general $L_0$ is not contained in $A_2\oplus A_0$. The last equality uses a formula in the proof of \cite[Lemma~18.2]{SV}.
\end{remark}

\begin{remark}\label{rem:final}
It also seems unlikely that the Beauville--Voisin
class $c_F(L_0)\in A_2\oplus A_0\subset\CH_0(F)$ can be
lifted to a Beauville--Voisin class $c_{L_0}\in \CH_0(F_{L_0})$. Indeed, as a consequence of Remark~\ref{rem:Proofrefined},
we know that
$$\psi\left(c_F(L_0)\right)|_{F_{L_0}}=2(g\cdot C-[L_0])\in \CH_0\left(F_{L_0}\right)
$$
with $C\in |\ko_F(1)|_{F_{L_0}}\otimes\pi_0^\ast\ko(-1)|$. So, if there exists
a Beauville--Voisin class in $\CH_0(F_{L_0})^+$ as in Remark~\ref{rem:Intro}, then
a natural guess would be that it must be proportional to $g\cdot C-[L_0]$.
Note, however, that the push-forward of $g\cdot C-[L_0]$ is typically not contained in $A_2\oplus A_0$, so not proportional to $c_F(L_0)$.
Even worse, the class $g\cdot C-[L_0]\in \CH_0(F_{L_0})$ is not even invariant. 
\end{remark}

%%%%%%%%%%%%%%%%%%%%%
% References
%%%%%%%%%%%%%%%%%%%%%

\end{document}